\documentclass{amsart}

\usepackage{mathtools, amssymb, amsthm}
\mathtoolsset{mathic=true}
\usepackage{microtype}
\usepackage{hyperref}

\newtheorem{theorem}{Theorem}

\begin{document}

\title{Pick-up Sticks and the Fibonacci Factorial}

\author[Sudbury]{Aidan Sudbury}
\address{School of Mathematics, Monash University, Australia}
\email{aidan.sudbury@monash.edu.au}

\author[Sun]{Arthur Sun}
\address{Department of Computer Science and Technology, University of Cambridge, UK}
\email{aws53@cam.ac.uk}

\author[Treeby]{David Treeby}
\address{School of Mathematics, Monash University, Australia}
\email{david.treeby@monash.edu.au}

\author[Wang]{Edward Wang}
\address{School of Mathematics and Statistics, University of Melbourne, Australia}
\email{edward.wang@student.unimelb.edu.au}

\begin{abstract}
  We present a variation of the broken stick problem in which $n$ stick lengths are sampled uniformly at random. We prove that the probability that no three sticks can form a triangle is the reciprocal of the product of the first $n$ Fibonacci numbers. Extensions to quadrilaterals and general $k$-gons are also discussed.
\end{abstract}

\maketitle

\section{Introduction}

Random stick problems have long revealed unexpected structure in probability and geometry. Here we study a variant in which $n$ stick lengths are chosen independently and uniformly from $[0,1]$, and we ask for the probability that no three of them can form the sides of a triangle. Surprisingly, this probability is given by \[
  P_n = \frac{1}{F_1 F_2 \cdots F_n},
\] the reciprocal of the product of the first $n$ Fibonacci numbers.

The same recursive method that yields this result extends naturally to higher polygons: forbidding quadrilaterals leads to the Tribonacci sequence, and in general, one can explore analogous relations for $k$-gons. The following sections trace its connection to classical broken-stick problems, develop the proof, and discuss these extensions.

\section{Picking up (not breaking) sticks}

The \emph{broken stick problem} asks: when a stick is broken at two random points, what is the probability that the three resulting segments can form a triangle?
This problem dates back at least to Lemoine~\cite{lemoine} in 1873, and is now considered a classic in geometric probability. Further popularized by Martin Gardner~\cite{gardner} in 1959, the problem has since inspired numerous generalizations. 

Andrea and G\'{o}mez~\cite{andrea} extended the problem, determining the probability that a stick broken into $n$ pieces can form an $n$-gon, while Verreault~\cite{verreault2022a, verreault2022b} and Mukerjee~\cite{mukerjee2024} have provided the most complete modern treatments, deriving the probability that a stick broken into $n$ pieces can form a $k$-gon. Verreault's earlier paper \cite{verreault2022a} developed an analytic approach to the ``for all'' version of the $k$-gon condition using order statistics and exponential spacings, whereas his later work addressed the ``there exists'' version combinatorially through MacMahon partition analysis. Mukerjee subsequently rederived the latter result using a shorter, calculus-based proof. Together, these studies established both the combinatorial and analytic frameworks for the broken-stick family of problems.

Here, we investigate the related \emph{pick-up stick problem}. It is worth clearly distinguishing this from the broken stick model. Unlike previously, we now seek the probability that no triangle is formed from $n$ sticks, each being independently sampled from the uniform distribution on $[0, 1]$.

In the classical setting the stick lengths are \emph{dependent} because they must sum to~1: their joint law is $\operatorname{Dirichlet}(1,\dots,1)$. By contrast, in the pick-up stick or \emph{stick-sampling} model, all stick lengths are independent. This changes the nature of the problem substantially: the segments are no longer constrained to sum to~1, and their lengths are statistically independent.

The two problems are intrinsically linked. Let $U_1,\dots,U_n$ be independent $\operatorname{Unif}[0,1]$ variables and let $S=\sum_{i=1}^n U_i$. Conditional on $S=s\le 1$, the normalized vector $(U_1/s,\dots,U_n/s)$ is uniformly distributed on the simplex $\{x\in[0,1]^n:\sum_i x_i=1\}$ and hence has the $\mathrm{Dirichlet}(1,\dots,1)$ law. This is identical to that of the broken stick problem.

When $s>1$, the bounds $U_i\le 1$ imply $U_i/s\le 1/s$, so $(U_1/s,\dots,U_n/s)$ is uniform on the \emph{truncated} simplex $\{x\in[0,1]^n:\sum_i x_i=1,\ x_i\le 1/s\}$, i.e.\ a truncation of $\operatorname{Dirichlet}(1,\dots,1)$. 

Unconditionally, since $S$ has the Irwin--Hall distribution, the law of $(U_1/s,\dots,U_n/s)$ is a mixture over $s$, which yields probabilities that differ from those in the classical broken-stick model.

In this pick-up stick model, Petersen and Tenner~\cite{petersen} considered the probability that $n$ such sticks can form an $n$-gon, showing that this occurs with probability $1-\frac{1}{(n-1)!}$. Equivalently, the probability that the $n$ sticks \emph{cannot} form an $n$-gon is $\frac{1}{(n-1)!}$.

In this paper we examine a novel variation within the pick-up stick model: what is the probability that \emph{no three} of the $n$ sampled sticks can form a triangle? The surprising result is that this probability is given by
\[
  \frac{1}{F_1F_2\cdots F_n},
\]
where $F_n$ denotes the $n$th Fibonacci number. The product $F_1 F_2 \cdots F_n$ is known as the $n$th \emph{Fibonorial}, whose appearance here is striking. This contrasts with the factorials in Petersen and Tenner's result; the connection between Fibonacci numbers and a natural geometric probability appears to be previously unrecognized.

Later, we extend our method to compute the probability that no four sampled sticks can form a quadrilateral, and discuss generalizations to configurations avoiding any $k$-gon.

\section{Main result}

We now turn to the proof of the result. Our approach uses the joint distribution of the order statistics of the independent stick lengths, analyzed through exponential spacings. Although this technique has appeared in broken-stick analyses by Verreault~\cite{verreault2022a} and Mukerjee~\cite{mukerjee2024}, its role here is different due to the absence of the fixed-sum constraint on the lengths. We state the result formally below.

\begin{theorem}\label{thm1}
  If \(n\) real numbers are chosen uniformly from the unit interval \([0, 1]\), the probability that no three of the numbers can form a triangle is \[
    \prod_{i = 1}^{n} \frac{1}{F_i}
  ,\] where \(F_n\) are the Fibonacci numbers defined by \(F_1 = 1\), \(F_2 = 1\), and \(F_{n} = F_{n-1} + F_{n-2}\) for \(n > 2\).
\end{theorem}
\begin{proof}
  Let $U_1, U_2, \dots, U_n$ be independent and identically distributed uniform random variables on $[0, 1]$, each representing the length of a stick. Their \emph{order statistics}, denoted $U_{(1)}, U_{(2)}, \dots, U_{(n)}$, are obtained by sorting them into an increasing sequence, such that \[
    U_{(1)} \le U_{(2)} \le \dots \le U_{(n)}
  .\] 
  The requirement that no three lengths form a triangle means that
  \[
    U_{(a)} + U_{(b)} \le U_{(c)}
  \]
  for all indices $a < b < c$. Because the ordered lengths satisfy $U_{(1)} \le U_{(2)} \le \cdots \le U_{(n)}$, it suffices to check this condition only for consecutive triples:
  \begin{equation}\label{eq1}
    U_{(i)} + U_{(i+1)} \le U_{(i+2)},
  \end{equation}
  for all $1 \le i \le n-2$.

  It is well known that the normalized cumulative sums of exponential variables follow the same distribution as the order statistics of a $\operatorname{Uniform}(0, 1)$ sample~\cite{johnson}. In other words, we may write \[
    U_{(1)} = \frac{X_1}{S}, \quad U_{(2)} = \frac{X_1 + X_2}{S}, \quad \dots , \quad U_{(n)} = \frac{X_1 + \dots + X_n}{S},
  \] where $S = X_1 + \dots + X_{n+1}$ for independent and identically distributed exponential variables $X_i \sim \operatorname{Exp}(1)$, each with density $f(x) = e^{-x}$.

  Substituting this into \eqref{eq1} yields 
  \begin{align*}
    \frac{X_1 + \dots + X_i}{S} + \frac{X_1 + \dots + X_{i+1}}{S} &\le \frac{X_1 + \dots + X_{i+2}}{S} \\
    X_1 + \dots + X_{i} &\le X_{i+2},
  \end{align*}
  where $1\le i\le n-2$. This amounts to ensuring all of 
  \begin{align*}
    X_1 &\le X_3 \\
    X_1 + X_2 &\le X_4 \\
              &\vdotswithin{\le} \\
    X_1 + \dots + X_{n-2} &\le X_n,
  \end{align*}
  hold simultaneously, with $X_1$ and $X_2$ being unbounded between $0$ and infinity. 

  The successive inequalities define a nested region in $\mathbb{R}^n$, and the required probability corresponds to the volume of this region under the exponential density.
  To evaluate it, we integrate successively from $X_n$ down to $X_1$, each step producing a factor determined by the recurrence relation that will lead to the Fibonacci pattern.

  Denoting $P_n$ as the probability that no triangle can be formed from $n$ sticks, repeatedly integrating $X_i$ while satisfying the above inequalities yields
  \begin{align*}
    P_n &= \int_{0}^{\infty} e^{-x_1} \int_{0}^{\infty} e^{-x_2} \int_{x_1}^{\infty} \dotsi \int_{x_1 + \dots + x_{n-2}}^{\infty} e^{-x_n} \, dx_n \cdots dx_{3}\, dx_2\, dx_1 \\
        &= \int_{0}^{\infty} \int_{0}^{\infty}\int_{x_1}^{\infty} \dotsi \int_{x_1 + \dots + x_{n-2}}^{\infty} e^{-(x_1 + \dots + x_n)} \, dx_n \cdots dx_{3}\, dx_2\, dx_1.
  \end{align*}
  The innermost integral evaluates to
  \begin{equation}\label{eq2}
    \int_{x_1 + \dots + x_{n-2}}^{\infty} e^{-(x_1 + \dots + x_n)}\, dx_n = e^{-(2(x_1 + \dots + x_{n-2}) + x_{n-1})} 
    .
  \end{equation}
  Each integration introduces coefficients that propagate forward according to the same recurrence as the Fibonacci sequence. To formalize this, let $I_i$ denote the integral obtained after $i$ integrations. For example, \eqref{eq2} corresponds to $I_1$. We can prove inductively that \[
    I_{i} = \frac{1}{F_{1} \cdots F_i}e^{-(F_{i+2}(x_1 + \dots + x_{n-i-1}) + F_{i+1}x_{n-i})}.
  \] The base case $i=1$ obviously holds, and the induction is completed with the fact that 
  \begin{align*}
    I_{i+1} &= \frac{1}{\prod_{k=1}^{i}}\int_{x_1 + \dots + x_{n-i-2}}^{\infty} I_{i} \, dx_{n-i} \\
            &= \frac{1}{\prod_{k = 1}^{i+1}} e^{-((F_{i+2}+F_{i+1})(x_1 + \dots + x_{n-i-2}) + F_{i+2}x_{n-i-1})}.
  \end{align*}
  We thus obtain
  \begin{align*}
    P_n &= \int_{0}^{\infty} \int_{0}^{\infty} I_{n-2}\, dx_2 \, dx_1 \\
        &= \int_{0}^{\infty} \int_{0}^{\infty} \frac{1}{F_1 \cdots F_{n - 2}} e^{-(F_{n}x_1 + F_{n-1} x_{2})} \, dx_2\, dx_1 \\
        &= \int_{0}^{\infty} \frac{1}{F_1 \cdots F_{n-1}} e^{-(F_nx_1)}\, dx_1 \\
        &= \frac{1}{F_1 \cdots F_n},
  \end{align*}
  which is the desired result. 
\end{proof}

While the algebraic structure of the integrals in the above proof resembles that used in the broken-stick analyses of Verreault and Mukerjee, the mechanism here is quite different: independence of lengths replaces the Dirichlet constraint, and this change transforms the factorial recursions of the classical model into the Fibonacci recursion observed above. The clean appearance of the Fibonorial therefore arises naturally from the geometry of the independent sample.

\section{Quadrilaterals and the Tribonacci extension}

The recursive method used above extends naturally to higher polygons. For quadrilaterals, the condition that no four of the sampled sticks can form a closed figure translates into a system of inequalities analogous to the triangle case, but involving sums of three lengths. The same exponential-spacing machinery applies, and the resulting recursion involves one additional predecessor term: instead of the two-term Fibonacci relation, the probability sequence now satisfies a three-term \emph{Tribonacci} relation.

\begin{theorem}
  If \(n\) real numbers are chosen uniformly from the unit interval \([0, 1]\), the probability that no four of the numbers can form a quadrilateral is \[
    \frac{1}{T_n - T_{n-2}} \prod_{i = 1}^{n-1} \frac{1}{T_i}
  ,\] where \(T_k\) are the \emph{Tribonacci} numbers defined by \(T_1 = 1\), \(T_2 = 1\), \(T_3 = 2\), \(T_{k} = T_{k-1} + T_{k-2} + T_{k-3}\) for \(k > 3\).
\end{theorem}
\begin{proof}
  Like before, the inequality can be reduced to simultaneously satisfying \[
    U_{(i)} + U_{(i+1)} + U_{(i+2)} \le U_{(i+3)}
  \] for all $1 \le i \le n-3$.

  We use the same method of rewriting the order statistics as sums of exponential variables, getting the inequality
  \begin{align*}
    \frac{X_1 + \dots + X_{i}}{S} + \frac{X_1 + \dots + X_{i+1}}{S} + \frac{X_1 + \dots + X_{i+2}}{S} &\le \frac{X_1 + \dots + X_{i+3}}{S} \\
    2(X_1 + \dots + X_i) + X_{i+1} &\le X_{i+3}
  \end{align*}
  which needs to be satisfied for all $i$, leaving $X_1, X_2$ and $X_3$ unbounded.

  Again, we integrate over all the variables to obtain the required probability, which we denote $P_n$. This gives us
  \begin{align*}
    P_n &= \int_{0}^{\infty} e^{-x_1}\int_{0}^{\infty} \dotsi \int_{2(x_1 + \dots + x_{n-3}) + x_{n-2}}^{\infty} e^{-x_n}\, dx_n\, \cdots\, dx_2\, dx_1, \\
        &= \int_{0}^{\infty} \int_{0}^{\infty} \dotsi \int_{2(x_1 + \dots + x_{n-3}) + x_{n-2}}^{\infty} e^{-(x_1 + \dots + x_n)}\, dx_n \, \cdots \, dx_2 \, dx_1.
  \end{align*}

  Let the integral obtained after $i$ integrations be denoted $I_i$. For example, we have \[
    I_1 = \int_{2(x_1+\dots+x_{n-3})+x_{n-2}}^{\infty} e^{-(x_1+\dots+x_n)}\, dx_n = e^{-(3(x_1+\dots+x_{n-3})+2x_{n-2}+x_{n-1})}
  .\]

  We now introduce three sequences, $(R_k)$, $(S_k)$ and $(T_k)$, to track the coefficients of the terms in the exponent during the inductive process of successive integrations. First, let $R_2=3$, $S_2=2$, $T_2=1$ so that the expression for $I_1$ above agrees with the $k=2$ case. We now derive recurrence relations from the coefficient updates in the exponent after each integration:
  \begin{align*}
    R_{k} &= R_{k-1} + 2T_{k-1}, \\
    S_k &= R_{k-1} + T_{k-1}, \\
    T_{k} &= S_{k-1}.
  \end{align*}
  This is seen inductively---the first integral $I_1$ evaluates to \[
    I_1 = e^{-(R_2(x_1 + \dots + x_{n-3}) + S_2 x_{n-2} + T_2 x_{n-1})},
  \] and we then assume that $I_i$ evaluates to
  \begin{equation} \label{eq3}
    I_i = \frac{1}{\prod_{k=1}^{i} T_k} e^{-(R_{i+1}(x_1 + \dots + x_{n-2-i}) + S_{i+1} x_{n-1-i} + T_{i+1}x_{n-i})}.
  \end{equation}
  From this inductive hypothesis, we obtain
  \begin{align*}
    I_{i+1} &= \int_{2(x_1+\dots+x_{n-3-i}) + x_{n-2-i}}^{\infty} I_i\, dx_{n-i} \\
            &= \frac{1}{\prod_{k=1}^{i+1} T_k} e^{-((R_{i+1} + 2T_{i+1})(x_1 + \dots + x_{n-3-i}) + (R_{i+1} + T_{i+1})x_{n-2-i} + S_{i+1}x_{n-1-i})},
  \end{align*}
  which closes the induction, showing that the recurrence relations hold. It remains to verify that $T_k$ indeed satisfies the Tribonacci recursion. Subtracting $S_k$ from $R_k$ gives us $R_k - S_k = T_{k-1}$, so $R_k = S_k + T_{k-1}$. From this, substitutions yield 
  \begin{align*}
    T_k = S_{k-1} &= R_{k-2} + T_{k-2} \\
                  &= (S_{k-2} + T_{k-3}) + T_{k-2} \\
                  &= T_{k-1} + T_{k-3} + T_{k-2},
  \end{align*}
  so our $T_k$ are exactly the Tribonacci numbers. Now, using \eqref{eq3}, we have \[
    I_{n-3} = \frac{1}{\prod_{k = 1}^{n-3} T_k} e^{-(R_{n-2}x_1 + S_{n-2}x_2 + T_{n-2}x_3)}
  ,\] leaving us with
  \begin{equation*}
    P_n = \int_{0}^{\infty} \int_{0}^{\infty} \int_{0}^{\infty} \frac{1}{\prod_{k = 1}^{n-3} T_k} e^{-(R_{n-2}x_1 + S_{n-2}x_2 + T_{n-2}x_3)} \, dx_3\, dx_2\, dx_1.
  \end{equation*}
  Evaluating this integral yields \[
    P_n = \frac{1}{T_1 \cdots T_{n-2} S_{n-2} R_{n-2}}
    .\] From the recursive definitions of $R$ and $S$ previously, we have that $S_{n-2} = T_{n-1}$, and that $R_{n-2} = T_{n} - T_{n-2}$, ultimately giving us \[
    P_n = \frac{1}{(T_n - T_{n-2})T_1\cdots T_{n-1}}
  ,\] as desired.
\end{proof}
Observe the presence of a ``correction term'' in the form of $T_n - T_{n-2}$. This disrupts an otherwise perfect generalization of Theorem~\ref{thm1}.

\section{An open problem}

Our results for triangles and quadrilaterals suggest a deeper underlying structure in the probability that no subset of $n$ independently sampled stick lengths can form a $k$-gon. While the triangular case leads cleanly to the Fibonacci law, the quadrilateral case leads to the Tribonacci sequence but requires a small correction term. Both results hint at a general pattern in which the avoidance of $k$-gons corresponds to a $(k-1)$-step Fibonacci-type recurrence.

Similar correction terms appear in the broken-stick analyses of Verreault and Mukerjee for $k \ge 4$, where analogous recursions produce generalized Fibonacci numbers. We invite the reader to extend this framework and find closed-form expressions for higher values of $k$:
\begin{quote}
  \emph{Given $n$ independent real numbers chosen uniformly from $[0,1]$, what is the probability that no $k$ of them can be arranged to form the side lengths of a $k$-gon?}
\end{quote}

Moreover, the appearance of the Fibonacci numbers raises a natural question: is there a way to derive the main result of Theorem~\ref{thm1} using a purely combinatorial argument? A direct counting argument yielding the same Fibonorial law would shed further light on the probabilistic structure uncovered here. We encourage the reader to search for such a proof.

\bibliographystyle{amsplain}
\bibliography{references.bib}

@article{lemoine,
  author    = {\'{E}mile Lemoine},
  title     = {Sur une question de probabilit{\'e}s},
  journal   = {Bulletin de la Soci{\'e}t{\'e} Math{\'e}matique de France},
  volume    = {1},
  year      = {1873},
  pages     = {39--40},
}

@article{gardner,
  author  = {Martin Gardner},
  title   = {Mathematical Games: Probability and Ambiguity},
  journal = {Scientific American},
  volume  = {201},
  number  = {4},
  year    = {1959},
  month   = {October},
}

@article{andrea,
  author = {Carlos D'Andrea and Emiliano G\'{o}mez},
  title = {The Broken Spaghetti Noodle},
  journal = {The American Mathematical Monthly},
  volume = {113},
  number = {6},
  pages = {555--557},
  year = {2006},
  publisher = {Taylor \& Francis},
  doi = {10.1080/00029890.2006.11920336},
  URL = { 
    https://doi.org/10.1080/00029890.2006.11920336
  },
  eprint = { 
    https://doi.org/10.1080/00029890.2006.11920336
  }
}

@article{verreault2022a,
  title = {On the probability of forming polygons from a broken stick},
  journal = {Statistics \& Probability Letters},
  volume = {180},
  pages = {109237},
  year = {2022},
  issn = {0167-7152},
  doi = {https://doi.org/10.1016/j.spl.2021.109237},
  url = {https://www.sciencedirect.com/science/article/pii/S0167715221001991},
  author = {William Verreault},
}

@article{verreault2022b,
  title = {MacMahon Partition Analysis: A discrete approach to broken stick problems},
  journal = {Journal of Combinatorial Theory, Series A},
  volume = {187},
  pages = {105571},
  year = {2022},
  issn = {0097-3165},
  doi = {https://doi.org/10.1016/j.jcta.2021.105571},
  url = {https://www.sciencedirect.com/science/article/pii/S0097316521001709},
  author = {William Verreault},
}

@article{mukerjee2024,
  author    = {Rahul Mukerjee},
  title     = {A Statistical Approach to Broken Stick Problems},
  journal   = {Communications in Statistics--Theory and Methods},
  volume    = {53},
  number    = {2},
  pages     = {430--443},
  year      = {2024},
  publisher = {Taylor \& Francis},
  doi       = {10.1080/03610926.2022.2109384}
}

@article{petersen,
  author = {T. Kyle Petersen and Bridget Eileen Tenner},
  title = {Broken Bricks and the Pick-up Sticks Problem},
  journal = {Mathematics Magazine},
  volume = {93},
  number = {3},
  pages = {175--185},
  year = {2020},
  publisher = {Taylor \& Francis},
  doi = {10.1080/0025570X.2020.1736888},
}

@book{johnson,
  author    = {Norman L. Johnson and Samuel Kotz},
  title     = {Continuous Univariate Distributions},
  publisher = {Houghton Mifflin},
  address   = {Boston, Massachusetts},
  volume    = {2},
  year      = {1970},
}

\end{document}